\documentclass[10pt,reqno]{amsart}
\usepackage[colorlinks=true,
pdfstartview=FitV, linkcolor=cyan, citecolor=magenta,
urlcolor=]{hyperref}

\usepackage{amsmath,amsfonts,latexsym,amssymb}
\usepackage{amssymb,amsfonts, amsmath}
\usepackage{mathrsfs}
\usepackage[latin1]{inputenc}
\usepackage[T1]{fontenc}
\usepackage{ae,aecompl}
\usepackage{braket}
\usepackage{comment}
\usepackage{amssymb,amsfonts, amsmath}
 \usepackage{mathtools} 
\usepackage[all,arc]{xy}
\usepackage{enumerate}
\usepackage{mathrsfs}
\usepackage{mathabx}
\usepackage{amssymb}
\usepackage{dsfont}
\usepackage[dvipsnames]{xcolor}
\usepackage{graphicx}
\usepackage{subfig}

\usepackage{verbatim} 
\usepackage{amssymb,amsfonts, amsmath}
\usepackage[all,arc]{xy}
\usepackage{enumerate}
\usepackage{mathrsfs}
\usepackage{hyperref}
\usepackage{xstring}
\usepackage{amsmath}
\usepackage[makeroom]{cancel}
\usepackage{cleveref}
\newtheorem{theorem}{Theorem}[section]
\newtheorem{lemma}[theorem]{Lemma}
\newtheorem{proposition}[theorem]{Proposition}

\newtheorem{observation}[theorem]{Observation}
\newtheorem{rmk}[theorem]{Remark}

\renewcommand{\leq}{\leqslant}
\renewcommand{\geq}{\geqslant}

\long\def\@savemarbox#1#2{\global\setbox#1\vtop{\hsize\marginparwidth 
  \@parboxrestore\tiny\raggedright #2}}
\usepackage{braket}
\usepackage{amsmath}
\usepackage[title]{appendix}

\makeatletter
\renewcommand*\env@matrix[1][\arraystretch]{%
\edef\arraystretch{#1}%
\hskip -\arraycolsep
\let\@ifnextchar\new@ifnextchar
\array{*\c@MaxMatrixCols c}}
\makeatother
\AtEndDocument{\bigskip{\footnotesize%
\textsc{Faculty of Mathematics, Karlsruhe Institute of Technology, Englerstra\ss e 2, 76131 Karlsruhe, Germany} \par
\textit{E-mail address}: \texttt{claudio.llosa@kit.edu}

 \medskip
 \textsc{Max Planck Institute for Mathematics in the Sciences, Inselstra{\ss}e 22, 04103 Leipzig, Germany} \par  
 \textit{E-mail address}: \texttt{konstantinos.tsouvalas@mis.mpg.de}}}
\title[]{Subgroups of CAT(0) groups, exotic finiteness properties and non-QI-embeddings into linear groups}

\author{Claudio Llosa Isenrich}
\author{Konstantinos Tsouvalas}


\begin{document}

\frenchspacing

\maketitle

\begin{abstract}
    For every positive integer $n$ we construct an example of a subgroup $L< G$ of a linear ${\rm CAT}(0)$ group $G$ such that $L$ is of finiteness type $\mathcal{F}_{n-1}$ and not $\mathcal{F}_n$, and $L$ does not admit a representation into $\mathsf{GL}_d(k)$ which is a quasi-isometric embedding for any local field $k$. We further prove that there is a faithful representation of $L$ into some $\mathsf{GL}_{\ell}(\mathbb{C})$ which is not the restriction of any representation of $G$. This generalises a family of fibre products of type $\mathcal{F}_2$ not $\mathcal{F}_3$ with these properties constructed by the second author.
\end{abstract}

\section{Introduction}

Some of the most studied classes of groups in geometric group theory are non-positively curved groups and linear groups. Important examples of non-positively curved groups include hyperbolic groups and ${\rm CAT}(0)$ groups. They are known to share many strong properties. Given a non-positively curved group, it is natural to ask if its subgroups inherit these properties. This has driven much interesting research recently. Similarly, linear groups share many strong properties and a natural geometric strengthening of linearity is to be a quasi-isometrically embedded subgroup of some general linear group. This raises the problem which linear groups can be quasi-isometrically embedded in some (potentially different) general linear group over a local field. Since the associated metric spaces (symmetric spaces and Bruhat--Tits buildings) are ${\rm CAT}(0)$, this problem seems particularly natural for ${\rm CAT}(0)$ groups and their subgroups.

Motivated by this theme and a recent heightened interest in the construction of pathological subgroups via exotic finiteness properties, the goal of this work is to show that there are subgroups of linear ${\rm CAT}(0)$ groups with arbitrary exotic finiteness properties which do not admit any representation into a general linear group over a local field which is a quasi-isometric embedding. In particular, we will prove that they can be chosen so that the ``enveloping'' ${\rm CAT}(0)$ group does admit such a representation. 

\begin{theorem}\label{thm:Main}
    For every positive integer $m$ there is a ${\rm CAT}(0)$ group $G$ with a subgroup $L<G$ with the following properties:
    \begin{enumerate}
        \item $L$ is of type $\mathcal{F}_{m-1}$ and not of type $\mathcal{F}_m$;
        \item $G$ admits a quasi-isometric embedding into some $\mathsf{GL}_{s}(\mathbb{C})$ for some $s\in \mathbb{N}$;
        \item there is no local field $k$ and $d\in \mathbb{N}$, for which $L$ admits a representation into $\mathsf{GL}_d(k)$ which is a quasi-isometric embedding; and
        \item there is $\ell\in \mathbb{N}$ and a faithful irreducible representation of $L$ into $\mathsf{GL}_{\ell}(\mathbb{C})$ which does not extend to a representation of $G$ into $\mathsf{GL}_{\ell}(\mathbb{C})$. 
    \end{enumerate}
\end{theorem}

This generalises work by the second author \cite{Tso-22} who proved this result for finitely presented subgroups of ${\rm CAT}(0)$ groups which are not $\mathcal{F}_3$. Here we say that a group is of \emph{finiteness type $\mathcal{F}_n$} if it admits a classifying space with a finite $n$-skeleton and that it has \emph{exotic finiteness properties} if it is $\mathcal{F}_{n-1}$ and not $\mathcal{F}_{n}$ for some $n>0$. A classical family of examples of groups of type $\mathcal{F}_{n-1}$ and not $\mathcal{F}_n$ for every $n>0$ are the Stallings--Bieri groups $SB_n$. By definition, $SB_n$ is the kernel of a factorwise surjective homomorphism $F_2 \times \cdots \times F_2 \to \mathbb{Z}$ from a direct product of $n$ $2$-generated free groups onto $\mathbb{Z}$ \cite{Sta-63,Bie-76}. In particular, the Stallings--Bieri groups are naturally subgroups of ${\rm CAT}(0)$ groups. 

To our knowledge, the examples in \Cref{thm:Main} are the first known examples of subgroups of $\textup{CAT}(0)$ groups with arbitrary exotic finiteness properties, all of whose representations into any general linear group over a local field fail to be quasi-isometric embeddings. If one drops the ${\rm CAT}(0)$ condition there is a second family known to experts. It is obtained by taking direct products of the fundamental group of the unit tangent bundle of a closed hyperbolic surface with the Stallings--Bieri groups (see \Cref{unit-tangent} for details).

Our proof of \Cref{thm:Main} relies on the recent construction of subgroups of cocompact complex hyperbolic lattices with exotic finiteness properties by the first author and Py \cite{LIPy-23,LIP-24}. More precisely, a key step in our proof is the following result, which may be of independent interest. 

\begin{theorem}\label{thm:Main-fibre-product}
    Let $n>1$ be an integer and let $\Gamma < \mathsf{PU}(n,1)$ be a cocompact lattice that admits a surjective homomorphism $\phi: \Gamma \to \mathbb{Z}^2$ with kernel of type $\mathcal{F}_{n-1}$ and not $\mathcal{F}_n$. Then the fibre product $P:=\big\{(g_1,g_2)\in \Gamma\times \Gamma \mid \phi(g_1)=\phi(g_2)\big\}< \Gamma \times \Gamma$ has the following properties:
    \begin{enumerate}
        \item $P$ is of type $\mathcal{F}_{2n-1}$ and not $\mathcal{F}_{2n}$; and
        \item there is no local field $k$ such that $P$ admits a representation into $\mathsf{GL}_d(k)$ which is a quasi-isometric embedding.
    \end{enumerate}
    Moreover, for every $n>1$, there exists a cocompact lattice $\Gamma<\mathsf{PU}(n,1)$, which satisfies the hypotheses and, in addition, can be chosen to be torsion-free and residually $p$-finite.
\end{theorem}
Concrete examples of lattices as in Theorem~\ref{thm:Main-fibre-product} are provided by any ``deep enough'' finite index subgroups of a cocompact lattice of the simplest type.
Since cocompact lattices in $\mathsf{PU}(n,1)$ are ${\rm CAT}(0)$, Theorem \ref{thm:Main-fibre-product} shows that there are subgroups of {\rm CAT}(0) groups which cover an infinite range of exotic finiteness properties and do not admit a quasi-isometrically embedded representation into $\mathsf{GL}_d(k)$ for any local field $k$. However, they do not span the full range of exotic finiteness properties, and, a priori, it could be true that each such representation extends to a representation of $\Gamma\times \Gamma$, which would make it distorted in a trivial way. To address these two points and thus deduce \Cref{thm:Main} from \Cref{thm:Main-fibre-product} we take a direct product of our examples with a Stallings--Bieri group and pass to a suitable finite index subgroup. 

Finally, one may wonder if being of type $\mathcal{F}_{n-1}$ and not $\mathcal{F}_n$ for some (sufficiently large) positive integer $n$ already implies that there is no representation into $\mathsf{GL}_d(k)$ for any local field $k$ which is a quasi-isometric embedding. It is easy to see that such a statement can not be true, even if the group is a subgroup of a ${\rm CAT}(0)$ group. 
Indeed, an example is given by the Stallings--Bieri groups. It is not hard to see that the canonical embedding $SB_n\hookrightarrow F_2\times \cdots \times F_2$ is a quasi-isometric embedding for all $n\geq 2$; for $n=2$ this is \cite[Theorem 2]{OlsSap-01} and for $n\geq 3$ this either follows by similar arguments or is implicit in Carter--Forester's proof that the Stallings--Bieri groups have quadratic Dehn function in \cite{CarFor-17}. Since $F_2\times \cdots \times F_2$ admits a quasi-isometric embedding into $\mathsf{GL}_{2n}(\mathbb{R})$ this implies that $SB_n$ does as well.

\subsection*{Structure} In \Cref{sec:Background-and-Preliminary} we provide some background on finiteness properties, subgroups of hyperbolic groups and Cartan projections, and prove some auxiliary results. \Cref{sec:non-extendability} contains non-extendability criteria for faithful embeddings into general linear groups over $\mathbb{C}$, that we will require to prove \Cref{thm:Main}(4). Finally, in \Cref{sec:Proof-of-Main-results}, we combine the results from the previous sections to prove \Cref{thm:Main,thm:Main-fibre-product} and discuss some related results and examples.

\subsection*{Acknowledgements} We would like to thank Sami Douba for interesting discussions and, in particular, for explaining to us the family of examples obtained from the unit tangent bundle of a closed hyperbolic surface in \Cref{unit-tangent}.

\section{Background and preliminary results}\label{sec:Background-and-Preliminary}
In this section we summarise relevant background material on finiteness properties, subgroups of hyperbolic groups and Cartan projections, and show some preliminary results that we will require to prove \Cref{thm:Main,thm:Main-fibre-product}.

\subsection{Finiteness properties and fibre products}
The finiteness properties $\mathcal{F}_n$ generalise being finitely generated, which is equivalent to $\mathcal{F}_1$, and being finitely presented, which is equivalent to $\mathcal{F}_2$. 

Given two short exact sequences 
\[
1\to N_1\to \Gamma_1 \stackrel{\phi_1}{\to} Q\to 1,
\]
\[
1\to N_2 \to \Gamma_2 \stackrel{\phi_2}{\to} Q \to 1,
\]
their (asymmetric) fibre product is the subgroup $P=\left\{(g_1,g_2)\mid \phi_1(g_1)=\phi_2(g_2)\right\}<\Gamma_1\times \Gamma_2$.

Fibre products are an important source of groups with exotic finiteness properties and interesting additional properties. A key reason for this are two results known as the 0-1-2 Lemma (e.g. \cite[Lemma 18]{Mil-71}) and the 1-2-3 Theorem \cite[Theorem B]{BHMS-13}, \cite{BriBMS}. They say that if $N_1$ and $N_2$ are $\mathcal{F}_0$ (resp. $\mathcal{F}_1$), $\Gamma_1$ and $\Gamma_2$ are $\mathcal{F}_1$ (resp. $\mathcal{F}_2$), and $Q$ is $\mathcal{F}_2$ (resp. $\mathcal{F}_3$), then $P$ is $\mathcal{F}_1$ (resp. $\mathcal{F}_2$). The assertion that the same is also true if we replace these finiteness properties by $\mathcal{F}_n$, $\mathcal{F}_{n+1}$ and $\mathcal{F}_{n+2}$ for $n\geq 3$ is known as the $n-(n+1)-(n+2)$-Conjecture. This conjecture remains open in the general case, but was confirmed in important special cases, including when $Q$ is abelian \cite[Theorem 6.3]{Kuc-14}. In fact, \cite[Theorem 6.3]{Kuc-14} allows for the $N_i$ to have distinct finiteness properties. We state Kuckuck's theorem in slightly higher generality, noting that his proof actually shows this more general case. We use the notation for the fibre product introduced above:
\begin{theorem}\label{thm:coabelian-n-n+1-n+2}
    Let $n,,k,\ell\geq 0$ with $n\geq k+\ell$. Assume that $\Gamma_1$ and $\Gamma_2$ are $\mathcal{F}_n$, $Q$ is abelian, $N_1$ is $\mathcal{F}_k$ and $N_2$ is $\mathcal{F}_{\ell}$. Then $P$ is $\mathcal{F}_{k+\ell+1}$.
\end{theorem}

We will also require the well-known result that finiteness properties behave well under direct products (see e.g. \cite[Section 7.2]{Geo-08}):
\begin{proposition}\label{product-typeF_m}
    Let $\Gamma_1$ and $\Gamma_2$ be groups and let $k\geq \ell \geq 0$ be integers. If $\Gamma_1$ is of type $\mathcal{F}_k$ and $\Gamma_2$ is of type $\mathcal{F}_{\ell}$ and not $\mathcal{F}_{\ell+1}$, then $\Gamma_1\times \Gamma_2$ is of type $\mathcal{F}_{\ell}$ and not $\mathcal{F}_{\ell+1}$.
\end{proposition}

\subsection{Subgroups of hyperbolic groups with exotic finiteness properties}
In recent work the first author and Py \cite{LIPy-23,LIP-24} constructed examples of subgroups of hyperbolic groups of type $\mathcal{F}_{n-1}$ and not $\mathcal{F}_n$ for all positive integers $n>0$, answering an old question of Brady. Their examples are kernels of morphisms from cocompact arithmetic lattices in $\mathsf{PU}(n,1)$ onto $\mathbb{Z}$ or $\mathbb{Z}^2$. We will require a straight-forward consequence of their result for  $\mathbb{Z}^2$. To state it we recall two definitions.

For a prime $p$, we say that a group $G$ is \emph{residually $p$-finite}, if for every element $g\in G\smallsetminus \left\{1\right\}$ there is a homomorphism $\phi: G\to Q$ to a finite $p$-group $Q$ such that $\phi(g)\neq 1$. 

\begin{theorem}\label{thm:LIP-existence}
    For $n\geq 1$ let $\Gamma<\mathsf{PU}(n,1)$ be a cocompact arithmetic lattice with virtually positive first Betti number. Then there exists a finite index torsion-free residually $p$-finite subgroup $\Gamma_0<\Gamma$ and a surjective homomorphism $\phi: \Gamma_0 \to \mathbb{Z}^2$ with kernel of type $\mathcal{F}_{n-1}$ and not $\mathcal{F}_n$.
\end{theorem}
\begin{proof}
    By \cite[Corollary 5.2]{LIPy-23} there exists a finite-index subgroup $\Gamma_0<\Gamma$ for which there is a surjective homomorphism $\phi : \Gamma_0 \to \mathbb{Z}^2$ with kernel of type $\mathcal{F}_{n-1}$ and not $\mathcal{F}_n$. The group $\Gamma_0$ is finitely generated linear and thus virtually torsion-free by Selberg's Lemma and virtually residually $p$-finite by a result of Platonov \cite{Platonov-68}. Moreover, these two properties, as well as finiteness properties are preserved under passage to finite index subgroups. Thus, restricting $\phi$ to a further finite index subgroup of $\Gamma_0$ and replacing $\mathbb{Z}^2$ by the image of this restriction, yields the desired conclusion.
\end{proof}

This result generalised previous constructions of finitely presented non-hyperbolic subgroups of hyperbolic groups constructed as kernels of morphisms onto $\mathbb{Z}$ \cite{Rip-82, Bra-99, Lod-18, Kro-21, IMM-23, LIP-24, LIMP-24}. For our construction it is important that the quotient group is the non-hyperbolic group $\mathbb{Z}^2$. Intuitively, this is because the non-linearity of its Dehn function implies that the kernel of this quotient morphism is distorted inside the fibre product.

\begin{rmk}\normalfont{Py and the first author also showed in \cite{LIPy-23} that the fundamental groups of the $n$-dimensional Stover--Toledo manifolds constructed as ramified covers of ball quotients in \cite{StoTol-22} admit morphisms onto $\mathbb{Z}^2$ with kernel of type $\mathcal{F}_{n-1}$, but not $\mathcal{F}_{n}$. Moreover, Kropholler and the first author produced examples of subgroups of hyperbolic groups of type $\mathcal{F}_2$ and not $\mathcal{F}_3$ that are kernels of homomorphisms onto $\mathbb{Z}^m$ for any $m\geq 1$ \cite{KroLlo-24}. Our proof of \Cref{thm:Main-fibre-product} can also be applied to produce fibre products of these groups that do not admit a quasi-isometrically embedded representation into $\mathsf{GL}_d(k)$ for any local field $k$.}\end{rmk}

\subsection{Cartan projection on $\mathsf{GL}_{d}(k)$ and quasi-isometric embeddings.} Throughout the remainder of this note $k$ denotes a local field, i.e $\mathbb{R}$, $\mathbb{C}$ or a finite extension of $\mathbb{Q}_p$ or the field of formal Laurent series $\mathbb{F}_q((T))$ over the finite field $\mathbb{F}_q$. We will now introduce the Cartan projection, distinguishing the Archimedean and non-Archimedean cases. For more details and background  we refer to \cite[Chapter 9]{Hel-78} in the Archimedean case and to  \cite{Bruhat-Tits} in the non-Archimedean case. Our exposition and notation closely follows the one in \cite[Section 2.2]{Tso-22}, see also \cite[Section 2]{Kas-08} and \cite[Section 2.1.2]{PozSamWie-21}.

Let $d\geq 2$ and $$\mathfrak{a}^{+}:=\big\{(x_1,\ldots,x_d) \in \mathbb{R}^d:x_1 \geq \ldots \geq x_d\big\}.$$ We equip the vector space $\mathbb{R}^d$ with the standard Euclidean norm $||\cdot||$.  The \emph{Cartan projection} $\mu : \mathsf{GL}_d(k)\to \mathfrak{a}^{+}$ is defined as follows.
\medskip

\noindent (a) {\em $k$ is Archimedean}. For a matrix $g=(g_{ij})_{i,j=1}^{d}$ let $g^{\ast}=(\overline{g_{ji}})_{i,j=1}^d$ be its conjugate transpose and $\ell_1(g)\geq \ell_2(g)\geq \cdots \geq \ell_d(g)$ the moduli of the eigenvalues of $g$ in non-increasing order. For $i=1,\ldots, d$, $\sigma_i(g):=\sqrt{\ell_i(gg^{\ast})}$ is the {\em $i$-th singular value of $g$}. The {\em Cartan projection} is the continuous, proper, surjective map $\mu:\mathsf{GL}_d(k)\rightarrow \mathfrak{a}^{+}$ defined by  $$\mu(g):=\big(\log\sigma_1(g),\ldots,\log \sigma_d(g) \big), \ g\in \mathsf{GL}_d(k).$$ 

\noindent (b) {\em $k$ is non-Archimedean.}  Let $\omega:k \rightarrow \mathbb{Z}\cup \{\infty\}$ be a discrete valuation on $k$ and consider the absolute value $|\cdot|=p^{-\omega(\cdot)}$, where $p\in \mathbb{N}$ is the cardinality of the residue field of $k$. Let $\mathcal{O}_k$ be the ring of integers of $k$ and fix $\tau \in \mathcal{O}_k$ a uniformizer such that $\omega(\tau)=1$.

 The group \hbox{$\mathsf{K}=\mathsf{GL}_d(\mathcal{O}_k)$} is the, up to conjugation, unique maximal compact subgroup of $\mathsf{GL}_d(k)$. A maximal $k$-split torus of $\mathsf{GL}_d(k)$ is the subgroup of diagonal matrices with entries in $k$ with positive Weyl chamber $$\mathsf{A}^{+}=\big \{\textup{diag}\big(a_1,\ldots, a_d\big): |a_1| \geq \ldots \geq |a_d|,\ a_i \in k^{\ast} \big\}.$$ The corresponding Cartan decomposition is $\mathsf{GL}_d(k)=\mathsf{K}\mathsf{A}^{+}\mathsf{K}$, i.e. every matrix $g \in \mathsf{GL}_d(k)$ can be written in the form $g=k_ga_g^{+}k_{g}'$, where $k_g,k_g' \in \mathsf{K}$ and $$a_g^{+} =\textup{diag}\big(\tau^{n_1(g)},\ldots, \tau^{n_d(g)}\big)\in \mathsf{A}^{+}$$ with $n_1(g),\ldots,n_d(g)\in \mathbb{Z}$ and $n_d(g) \geq \ldots \geq n_1(g)$. For $1 \leq i \leq d$, the \hbox{{\em $i$-th singular value of $g$} is} $$\sigma_i(g):=\big|\tau^{n_i(g)}\big|=p^{-n_i(g)}.$$ The Cartan projection $\mu:\mathsf{GL}_d(k) \rightarrow \mathfrak{a}^{+}$ is the map $$\mu(g):=\big(\log \sigma_1(g),\ldots, \log \sigma_d(g) \big), \ g \in \mathsf{GL}_d(k).$$

Let $\mathsf{H}$ be a finitely generated group equipped with a left invariant word metric $d_{\mathsf{H}}$ and $|\cdot|_{\mathsf{H}}:\mathsf{H}\rightarrow \mathbb{R}_{+}$ be the word length function of $d_{\mathsf{H}}$, where $|h|_{\mathsf{H}}=d_{\mathsf{H}}(h,1)$, $h \in \mathsf{H}$. A linear representation $\rho:\mathsf{H}\rightarrow \mathsf{GL}_d(k)$ is called a {\em quasi-isometric embedding} if there exist $C,c>1$, such that $$c^{-1}|h|_{\mathsf{H}}-C\leq \big|\big|\mu(\rho(h))\big|\big|\leq c|h|_{\mathsf{H}}+C \ \ \forall h\in \mathsf{H}.$$

If $k=\mathbb{R}$ or $\mathbb{C}$ (resp. $k$ is non-Archimedean), let $(X,\mathsf{d})$ be the Riemannian symmetric space $\mathsf{GL}_d(k)/K$ (resp. the Bruhat--Tits building on which $\mathsf{GL}_d(k)$ acts properly by isometries) equipped with the Killing metric $\mathsf{d}$ (resp. the metric $\mathsf{d}$ restricting to the Euclidean metric on each apartment of $X$). A linear representation $\rho:\mathsf{H}\rightarrow \mathsf{GL}_d(k)$ is a quasi-isometric embedding if and only if any orbit map $o_{\rho,x_0}:(\mathsf{H},d_{\mathsf{H}})\rightarrow (X,\mathsf{d})$, $o_{\rho,x_0}(\gamma)=\rho(\gamma)x_0, \gamma \in \mathsf{H}$, is a quasi-isometric embedding of metric spaces (see e.g. \cite[Section 2.4]{Kas-08}).

For $h_1,h_2\in \mathsf{H}$ define the commutator $[h_1,h_2]=h_1h_2h_1^{-1}h_2^{-1}$. Inductively, for $h_1,\ldots,h_r\in \mathsf{H}$, their commutator is $[h_1,\ldots,h_r]:=\big[[h_1,\ldots,h_{r-1}],h_r\big]$. We recall here the following upper bound from \cite{Tso-22} for the norm of the Cartan projection of representations of fiber products over local fields.

\begin{theorem}\textup{(\cite[Corollary 1.2]{Tso-22})}\label{estimate-1} Let $\mathsf{\Gamma}$ be a finitely generated group and fix a word length function $|\cdot|_{\mathsf{\Gamma}}:\mathsf{\Gamma}\rightarrow \mathbb{R}_{+}$. Suppose that $\mathsf{N}$ is a normal subgroup of $\mathsf{\Gamma}$, let $\phi:\mathsf{\Gamma}\rightarrow \mathsf{\Gamma}/\mathsf{N}$ be the natural projection and $$P:=\big\{(g_1,g_2)\in \mathsf{\Gamma}\times \mathsf{\Gamma}: \phi(g_1)=\phi(g_2)\big\}.$$ For every representation $\rho:P \rightarrow \mathsf{GL}_d(k)$ over a local field $k$, there exists $C>0$, depending only on $\rho$, such that $$\big|\big|\mu(\rho(1,[h_1,\ldots,h_r]))\big|\big|\leq 2^rC\left(\sum_{i=1}^r\big|h_i\big|_{\mathsf{\Gamma}}\right)+C$$ for every $h_1,\ldots,h_r\in \mathsf{N}$ with $r\geq d+1$. \end{theorem}

As a consequence of Theorem \ref{estimate-1}, we obtain the following theorem that we use for the proof of \Cref{thm:Main-fibre-product}.

 \begin{theorem}\label{thm:Main-fibre-product-2}
    Let $\mathsf{\Gamma}$ be a Gromov hyperbolic group that admits a surjective homomorphism $\phi: \mathsf{\Gamma} \to \mathbb{Z}^m$, $m\geq 2$, with finitely generated kernel. Consider the fiber product $$P:=\big\{(g_1,g_2)\in \mathsf{\Gamma} \times \mathsf{\Gamma} \mid \phi(g_1)=\phi(g_2)\big\}$$ and fix a left invariant word metric $|\cdot|_P:P\rightarrow \mathbb{R}_{+}$. For every $d\in \mathbb{N}$, there exists a sequence $(w_n)_{n\in \mathbb{N}}$ in $P$, depending only on $d\in \mathbb{N}$, with the following properties: 
    \begin{enumerate}
        \item $|w_n|_P\to \infty$ as $n\to \infty$;
        \item for any local field $k$ and any representation $\rho:P\rightarrow \mathsf{GL}_d(k)$, there exists $C_{\rho}>0$ such that for every $n\in \mathbb{N}$, $$\big|\big|\mu(\rho(w_n))\big|\big|\leq C_{\rho}\big|w_n\big|_P^{1-\frac{1}{2d+3}}.$$
    \end{enumerate} 
    \end{theorem}
   
   \begin{proof} Given the presentation $\mathbb{Z}^m=\big\langle x_1,x_2,\ldots,x_m \ | \ [x_i,x_j], i,j=1,\ldots,k\big\rangle$ and a finite generating set $X:=\{a_1,\ldots, a_r\}$ of $\textup{ker}(\phi)$ there is a presentation (see e.g. \cite[Section 1.4]{BriBMS}) of the hyperbolic group $\mathsf{\Gamma}$ of the form: \begin{align}\label{presentation}\mathsf{\Gamma}=\Big \langle x_1,\ldots,x_m,a_1,\ldots, a_r \ \big | \ [x_1,x_2]W_{12}, \ldots, [x_{m-1},x_{m}]W_{(m-1)m}, W_1',\ldots, W_q'\Big\rangle\end{align} where $W_{12},\ldots, W_{m(m-1)},W_1',\ldots,W_q'$ are words in the normal closure of the set $X$ in the free group $F(X\cup\{x_1,\ldots,x_k\})$ generated by $X\cup \{x_1,\ldots,x_k\}$. Let $\overline{x_1},\overline{x_2}\in \mathsf{\Gamma}$ be the image of $x_1,x_2$ in $\mathsf{\Gamma}$. Since $[{x_1}^n,{x_2}^n]$ lies in the normal closure of $[x_1,x_2]$ in $F(X)$, by applying \cite[Lemma 4.7]{Tso-22} for the presentation (\ref{presentation}) of $\mathsf{\Gamma}$, for the sequence $$w_n:=\Bigg (1, \Bigg[\Big[\big[\overline{x_2}^n,\overline{x_1}^n \big],\overline{x_1}^n \Big],\underbrace{[\overline{x_2}^n,\overline{x_1}^n],\ldots, [\overline{x_2}^n,\overline{x_1}^n]}_{d-\textup{times}}\Bigg] \Bigg)$$ in $\mathsf{\Gamma}\times \mathsf{\Gamma}$, there is $C_1>0$, independent of $n$, such that \begin{align}\label{ineq-1}\left|w_n\right|_P\geq C_1n^{1+\frac{1}{2d+2}}, \ \forall n\in \mathbb{N}.\end{align}

   Now by applying Theorem \ref{estimate-1} for the representation $\rho:P\rightarrow \mathsf{GL}_d(k)$ and the sequence $(w_n)_{n\in \mathbb{N}}$, there is $C_2>0$, depending only on $\rho$, such that $$\big|\big|\mu(\rho(w_n))\big|\big|\leq C_{2}n, \ \forall n \in \mathbb{N}.$$ In particular, by using (\ref{ineq-1}), there is a constant $C_{\rho}>0$, depending only on $\rho$, such that: $$\big|\big|\mu(\rho(w_n))\big|\big|\leq C_{\rho}\big|w_n\big|_P^{1-\frac{1}{2d+3}}, \ \forall n \in \mathbb{N}.$$ This completes the proof of the theorem.\end{proof}

\section{Non-extendability of representations on finite index subgroups of fibre products}
\label{sec:non-extendability}

In this section we prove non-extendability criteria for faithful embeddings into $\mathsf{GL}_{\ell}(\mathbb{C})$ of suitable finite index subgroups of coabelian fibre products. 

Recall that for a group $G$ its lower central series is the descending series $(\gamma_i(G))_{i\geq 1}$, defined by $\gamma_1(G):= G$ and $\gamma_i(G):= [G,\gamma_{i-1}(G)]$ for $i\geq 2$. A group is nilpotent, if there is a $c\geq 1$ such that $\gamma_c(G)=\left\{1\right\}$. Finite $p$-groups are nilpotent. Thus, a residually $p$-finite group is also residually (finite) nilpotent. We use this to prove the following auxiliary lemma.

\begin{lemma}\label{lem:cyclic-quotient}
    Let $\Gamma$ be a group which is residually $p$-finite for some prime $p$, let $\phi : \Gamma \to A$ be a homomorphism onto an abelian group $A$, and let $N=\ker(\phi)$. Then for all elements $g,h\in \Gamma$ with $[g,h]\neq 1$ there is a finite index normal subgroup $N'< N$, such that $N'$ is normal in $\Gamma$, $[g,h]\in N'$, and the projection of $[g,h]$ to $N'/[N',N']$ is non-trivial.
\end{lemma}

\begin{proof}
    Since $N$ is the kernel of a homomorphism of $\Gamma$ onto an abelian group, it contains the commutator subgroup $\left[\Gamma,\Gamma\right]$ of $\Gamma$. Thus, for all elements $g,h\in \Gamma$, we have $[g,h]\in N$. 
    
    Since $\Gamma$ is residually $p$-finite, the normal subgroup $N$ of $\Gamma$ is also residually $p$-finite. Since $[g,h]\neq 1$ there exists a homomorphism $\psi: N \to Q$ onto a finite $p$-group such that $\psi([g,h])\neq 1$. Since $[g,h]\in \gamma_1(N)=N$ and finite $p$-groups are nilpotent, this is equivalent to saying that there is some $\ell\geq 1$ such that $[g,h]\in \ker(\psi) \cdot \gamma_{\ell}(N)$, but $[g,h]\notin \ker(\psi)\cdot \gamma_{\ell+1}(N)$. If we only want the subgroup $N'$ to be normal in $N$, then we can choose $N':=\ker(\psi)\cdot \gamma_{\ell}(N)$, since the projection of $[g,h]$ survives in the abelian group $(\ker(\psi) \cdot \gamma_{\ell}(N))/ (\ker(\psi)\cdot \gamma_{\ell+1}(N))\cong \gamma_{\ell}(\psi(N))/\gamma_{\ell+1}(\psi(N))$.
    
    By intersecting all subgroups of index $|Q|$ in $N$, we can find a characteristic finite index subgroup $N_0\unlhd N$ with $N_0\leq \ker(\psi)$. Since the lower central series terms are also characteristic subgroups of $N$, the descending sequence of subgroups $N_0\cdot \gamma_{\ell}(N)\leq N$ is characteristic. Moreover, $[g,h]\in N_0\cdot \gamma_1(N)=N$ and $N_0\cdot \gamma_{\ell}(N)\leq \ker(\psi) \cdot \gamma_{\ell}(N)$ implies that there is $\ell'\leq \ell$ such that $[g,h]\in N_0\cdot \gamma_{\ell'}(N)$ and $[g,h]\notin N_0\cdot \gamma_{\ell'+1}(N)$. Choosing $N':= N_0\cdot \gamma_{\ell'}(N)$ and again using that $(N_0\cdot \gamma_{\ell'}(N))/(N_0\cdot \gamma_{\ell'+1}(N))\cong \gamma_{\ell'}(N/N_0) / \gamma_{\ell'+1}(N/N_0)$ is abelian completes the proof, since characteristic subgroups of $N$ are normal in $\Gamma$.\end{proof}

\begin{rmk}\label{rmk:linear-p-finite}
\normalfont{By a theorem of Platonov \cite{Platonov-68} and Selberg's Lemma, every finitely generated linear group has a finite index subgroup which is residually $p$-finite for some prime $p$ and this finite index subgroup can be chosen to be torsion-free. In particular, we can apply Lemma \ref{lem:cyclic-quotient} to a finite index subgroup of every linear group.}
\end{rmk}
A subgroup $G<\mathsf{GL}_d(\mathbb{C})$ is called \emph{strongly irreducible}, if every finite-index subgroup $H<G$ acts irreducibly on $\mathbb{C}^d$. An element $g\in \mathsf{GL}_d(\mathbb{C})$ is called {\em proximal}, if there exists a unique attracting fixed point $x_{g}^{+}\in \mathbb{P}(\mathbb{C}^d)$ and a repelling hyperplane $V_{g}^{-}$, with $x_{g}^{+}\oplus V_{g}^{-}=\mathbb{C}^d$, such that $\lim_n g^nx=x_{g}^{+}$, for every $x\in \mathbb{P}(\mathbb{C}^d)\smallsetminus \mathbb{P}(V_{g}^{-})$; equivalently, $g$ has a unique eigenvalue of maximal absolute value. A subgroup $G<\mathsf{GL}_d(\mathbb{C})$ is \emph{proximal} if it contains a proximal element; denote by $\Lambda_G$ its proximal limit set, i.e. the closure of attracting fixed points of proximal elements of $G$ in $\mathbb{P}(\mathbb{C}^d)$. If $G$ is irreducible, then there is a basis of attracting fixed points of proximal elements of $G$ and thus $\Lambda_G$ spans $\mathbb{C}^d$.\footnote{To see this, observe that $\Lambda_G$ is invariant under the $G$-action on $\mathbb{P}(\mathbb{C}^d)$, since for $g\in G$ proximal and $h\in G$ the point $h\cdot x_g^+$ is the unique attracting fixed point of $hgh^{-1}\in G$.}

We will need the following observation.

\begin{observation}\label{irreducible} \textup{(i)} Let $G_i<\mathsf{GL}_{d_i}(\mathbb{C})$, $d_i\geq 2$, $i=1,\ldots,r$, be irreducible subgroups containing a proximal element. Then the group $G_1\otimes \cdots \otimes G_r=\{g_1\otimes\cdots \otimes g_r:g_1,\ldots,g_r\in G\}$ is an irreducible subgroup of $\mathsf{GL}(\mathbb{C}^{d_1}\otimes \cdots \otimes \mathbb{C}^{d_r})$.\\
\noindent \textup{(ii)} Let $G<\mathsf{GL}_d(\mathbb{C})$ be a strongly irreducible subgroup and $N\unlhd G$ a normal subgroup containing a proximal element. Then $N<\mathsf{GL}_d(\mathbb{C})$ is also irreducible.\end{observation}

\begin{proof} (i) Let $G_{i}^{\ast}:=\{g^{\ast}:g\in G\}$ be the dual subgroup of $G_i$, which is also proximal and irreducible. Observe that for $g_i\in G_i$ proximal elements, $g_1\otimes \cdots \otimes g_r\in G_1\otimes \cdots \otimes G_r$ is proximal with attracting fixed point $[v_{g_1}\otimes \cdots \otimes v_{g_r}]$, $x_{g_i}^{+}=[v_{g_i}]$, $i=1,\ldots,r$. In particular, since $\Lambda_{G_i}$ spans $\mathbb{C}^{d_i}$, we conclude that $\Lambda_{G_1}\otimes \cdots \otimes \Lambda_{G_r}=\big\{[v_1\otimes \cdots \otimes v_r]:[v_i]\in \Lambda_{G_i}\big\}$ spans $\mathbb{C}^{d_1}\otimes \cdots \otimes \mathbb{C}^{d_r}$.

Similarly, since $G_i^{\ast}$ is proximal and irreducible, $\Lambda_{G_1^{\ast}}\otimes \cdots \otimes \Lambda_{G_r^{\ast}}$ also spans $\mathbb{C}^{d_1}\otimes \cdots \otimes\mathbb{C}^{d_r}$, which equivalently implies $\bigcap\big\{V_{g_1\otimes \cdots \otimes g_r}^{-}: g_i\in G_i \ \textup{proximal}\big\}=(0)$.\footnote{Here we use that for an attracting fixed point $x_{g_i^{\ast}}^+=[v_{g_i^{\ast}}]$ of a proximal element $g_i^{\ast}\in G_i^{\ast}$, we have $v_{g_i^{\ast}}^{\perp} = V_{g_i}^{-}$, where the orthogonal complement $\perp$ is with respect to the \hbox{standard Hermitian inner product on $\mathbb{C}^d$.}} Thus, if $(0)\neq V_0\subset \mathbb{C}^{d_1}\otimes \cdots \otimes \mathbb{C}^{d_r}$ is a $G_1\otimes \cdots \otimes G_r$-invariant subspace, there are $g_1,\ldots, g_r\in G$ proximal and $v_0\in V_0$ such that $v_0\notin V_{g_1\otimes \cdots \otimes g_r}^{-}$. Applying the $G_1\otimes \cdots \otimes G_r$-invariance of $V_0$ twice implies that for $x_{g_i}^{+}=[v_{g_i}]$ we have $[v_{g_1}\otimes \cdots \otimes v_{g_r}]=\lim_n (g_1^n \otimes \cdots \otimes g_r^n)v_0 \in \mathbb{P}(V_0)$ and thus for every $h_i\in G_i$, $h_1v_{g_1}\otimes \cdots \otimes h_rv_{g_r}\in V_0$. We deduce from the irreducibility of the $G_i$ that $V_0=\mathbb{C}^{d_1}\otimes \cdots \otimes \mathbb{C}^{d_r}$.\\

\noindent \textup{(ii)} Let us assume that $(0)\neq W\subset \mathbb{C}^d$ is an $N$-invariant subspace, of minimal dimension, on which $N$ acts irreducibly. Then, using that $N\unlhd G$ and induction, we may find $g_1,\ldots,g_{s}\in G$ such that $\mathbb{C}^d=\bigoplus_{i=1}^{s}g_iW$. Since attracting fixed points of proximal elements are unique, there is a non-empty subset $J\subset \{1,\ldots,s\}$ such that $\Lambda_N=\bigsqcup_{j\in J}\Lambda_N\cap \mathbb{P}(g_iW)$ and $\Lambda_N\cap \mathbb{P}(g_jW)$ is non-empty for every $j\in J$. Note that $\Lambda_N$ is $G$-invariant and $\Lambda_N\cap \mathbb{P}(g_jW)$ is acted on minimally by $N$ (since $N$ restricted to $g_jW$ acts irreducibly). Hence, every $g\in G$ has to permute the sets $\big \{\Lambda_N\cap \mathbb{P}(g_1W),\ldots, \Lambda_N\cap \mathbb{P}(g_sW)\big \}$. Thus, by passing to a finite-index subgroup $G'\unlhd G$, $G'$ has to preserve $\Lambda_N\cap \mathbb{P}(g_jW)$ for every $j$ and since $G'$ is irreducible we have $\Lambda_{G'}\subset \mathbb{P}(g_jW)$. This shows that $\textup{dim}(W)=d$ and hence $N$ is irreducible.\end{proof}

\begin{lemma}\label{non-extension}\noindent \textup{(i)} Let $\Delta<\mathsf{GL}_{d}(\mathbb{C})$ be a strongly irreducible, finitely generated subgroup, whose infinite order elements are proximal, and $K\unlhd \Delta$ an infinite finitely generated normal subgroup such that $\Delta/K$ is abelian. There exist finite-index subgroups $\Gamma\unlhd \Delta$ and $K'\unlhd K$ such that $K'$ is coabelian in $\Gamma$, and a finite-index subgroup $L$ of the fibre product $\left\{(g_1,g_2)\in \Gamma \times \Gamma: g_1^{-1}g_2\in K'\right\}$, such that there is an irreducible representation $\rho:L \rightarrow \mathsf{GL}_{r}(\mathbb{C})$, $r=d^2$,  which cannot extend to a representation of $\Gamma \times \Gamma$ into $\mathsf{GL}_{r}(\mathbb{C})$.\\
\noindent \textup{(ii)} Let $\Delta'<\mathsf{GL}_m(\mathbb{C})$ be a strongly irreducible, finitely generated subgroup, whose infinite order elements are proximal and $M$ a subgroup of the direct product of $q$ copies $\Delta' \times \cdots \times \Delta'$, containing $K'\times \cdots \times K'$, for some non-trivial infinite normal subgroup $K'\unlhd \Delta'$. Then there is an irreducible representation $\psi:L\times M \rightarrow \mathsf{GL}_{s}(\mathbb{C})$, $s=d^2m^q$, which does not extend to a representation of $\Gamma \times \Gamma \times \Delta' \times \cdots \times \Delta'$ into $\mathsf{GL}_s(\mathbb{C})$.
\end{lemma}

\begin{proof} (i) By \Cref{rmk:linear-p-finite}, we may find a prime $p\in \mathbb{N}$ and a finite-index subgroup $\Gamma \unlhd \Delta$ such that $\Gamma$ is torsion-free and residually $p$-finite, and $K':=K\cap \Gamma$ is a coabelian subgroup of $\Gamma$. Since $[\Gamma,\Gamma]$ is a subgroup of $K'$, the fiber product $P:=\left\{(g_1,g_2)\in \Gamma \times \Gamma: g_1^{-1}g_2\in K\right\}$ is a normal subgroup of $\Gamma \times \Gamma$ and the quotient $A:=(\Gamma \times \Gamma)/P$ is abelian.

Observe that $\Gamma\times \Gamma$ is residually $p$-finite and let $\phi:\Gamma \times \Gamma \rightarrow A$ be the natural projection. Fix two non-commuting elements $a,b\in \Gamma$ such that $\left(1,[a,b]\right)\in \Gamma \times \Gamma\smallsetminus \{(1,1)\}$. By applying Lemma \ref{lem:cyclic-quotient} for the group $\Gamma \times \Gamma$, the coabelian subgroup $P$ and the homomorphism $\phi$, we obtain a finite-index subgroup $L<P$ with $(1,[a,b])\in L$, such that the image of $(1,[a,b])$ in $L/[L,L]$ is non-trivial. In particular, there is a homomorphism $\pi:L\rightarrow \mathsf{GL}_1(\mathbb{C})$ such that $\pi\left((1,[a,b])\right)\neq 1$.

The strong irreducibility of $\Delta<\mathsf{GL}_d(\mathbb{C})$ implies that $\Gamma<\mathsf{GL}_d(\mathbb{C})$ is also strongly irreducible. Since every infinite order element of $\Gamma$ is proximal, every non-trivial normal subgroup $N \unlhd \Gamma$ contains a proximal element. Using strong irreducibility of $\Gamma$, by Observation \ref{irreducible} (ii), $N<\mathsf{GL}_d(\mathbb{C})$ is irreducible. Now consider the tensor product $\rho_0:\Gamma \times \Gamma \rightarrow \mathsf{GL}(\mathbb{C}^{d}\otimes \mathbb{C}^{d})$, $$\rho_0(g_1,g_2)=g_1\otimes g_2,\ (g_1,g_2)\in \Gamma \times \Gamma,$$ and define a representation $\rho:L\rightarrow \mathsf{GL}(\mathbb{C}^{d}\otimes \mathbb{C}^{d})$ as follows $$\rho(h_1,h_2)=\pi(h_1,h_2)\rho_0(h_1,h_2), \ (h_1,h_2)\in L.$$

We claim that $\rho$ is an irreducible representation that does not extend to a representation of $\Gamma \times \Gamma$. We first prove irreducibility by showing that $\rho$ restricts to an irreducible representation on a suitable subgroup. Since $L$ and $\textup{ker}(\pi)$ are finite-index subgroups of $P$, there is a finite-index subgroup $K_0<K'$ such that $K_0\times K_0$ is a normal subgroup of $\Gamma \times \Gamma$ contained in $\textup{ker}(\pi)$. In particular, $K_0<\mathsf{GL}_{d}(\mathbb{C})$ is irreducible and proximal, and $K_0\unlhd\Gamma$ is normal. Thus, Observation \ref{irreducible} (i) implies that  $\rho_0$ (and $\rho$) restricted to $K_0\times K_0$ is irreducible. 

Now suppose that there exists a representation $\overline{\rho}:\Gamma \times \Gamma \rightarrow \mathsf{GL}(\mathbb{C}^{d}\otimes \mathbb{C}^{d})$ extending $\rho$. Observe that for every $(g_1,g_2)\in \Gamma \times \Gamma$ and $(h_1,h_2)\in K_0\times K_0$, we have $(g_1h_1g_1^{-1},g_2h_2g_2^{-1})\in \textup{ker}(\pi)$ and hence \begin{align*}\overline{\rho}(g_1,g_2)\rho(h_1,h_2)\overline{\rho}(g_1,g_2)^{-1}&=\rho(g_1h_1g_1^{-1},g_2h_2g_2^{-1})\\ &=\rho_0(g_1h_1g_1^{-1},g_2h_2g_2^{-1})\\ &=\rho_0(g_1,g_2)\rho(h_1,h_2)\rho_0(g_1,g_2)^{-1}.\end{align*} 

Thus, $\rho_0(g_1,g_2)^{-1}\overline{\rho}(g_1,g_2)$ centralizes $\rho(K_0\times K_0)$, for every $(g_1,g_2)\in \Gamma \times \Gamma$. Since $\rho(K_0\times K_0)$ is irreducible, we deduce that there is $\lambda(g_1,g_2)\in \mathbb{C}^{\ast}$ such that $$\overline{\rho}(g_1,g_2)=\lambda(g_1,g_2)\rho_0(g_1,g_2).$$ Since $\overline{\rho}$ and $\rho_0$ are group homomorphisms, the map $\lambda:\Gamma \times \Gamma \rightarrow \mathsf{GL}_1(\mathbb{C})$, $(g_1,g_2)\mapsto \lambda(g_1,g_2)$, is a group homomorphism. In particular, since $\overline{\rho}$ is an extension of $\rho$, we have that $$\lambda(g_1,g_2)=\pi(g_1,g_2), \ \mbox{ for all } (g_1,g_2)\in L.$$ However, by the choice of $\pi$, $\pi((1,[a,b]))\neq 1$ while $\lambda((1,[a,b]))=[\lambda(1,a),\lambda(1,b)]=1$. This is a contradiction, hence the conclusion follows.
\medskip

\noindent \textup{(ii)} Let $\pi:L\rightarrow \mathsf{GL}_1(\mathbb{C})$ be the homomorphism from (i) such that $\pi((1,[a,b]))\neq 1$. Define a representation $\rho_0':\Gamma \times \Gamma \times \Delta'\times \cdots \times \Delta'\rightarrow \mathsf{GL}(W)$, $W:=\mathbb{C}^d\otimes \mathbb{C}^d\otimes \mathbb{C}^m\otimes \cdots \otimes \mathbb{C}^m$, by $$\rho_0'(g_1,g_2,g_3,\ldots,g_{q+2})=g_1\otimes g_2\otimes g_3\otimes \cdots \otimes g_{q+2}$$ for $g_1,g_2\in \Gamma, g_3,\ldots,g_q\in \Delta'$. 

We claim that the representation $\rho':L\times M\rightarrow \mathsf{GL}(W)$, $$\rho'(g_1,g_2,g)=\pi(g_1,g_2)\rho_0'(g_1,g_2,g), \ (g_1,g_2)\in L, \ g\in M,$$ is irreducible and does not extend to a representation of $\Gamma \times \Gamma \times M$ into $\mathsf{GL}(W)$. The proof is similar to the proof of (i).

We first prove irreducibility. Let $K_0\times K_0$ be the subgroup of $L$ from (i) such that $K_0\times K_0<\textup{ker}(\pi)$, and $K'\unlhd M$, such that $K'\times \cdots \times K' <M$. As in (i), $K'<\mathsf{GL}_m(\mathbb{C})$ is irreducible and proximal. Thus, Observation \ref{irreducible} (ii) implies that $\rho_0'$ (and $\rho'$) restricted to $K_0\times K_0 \times K'\times \cdots \times K'$ is irreducible. 

Now suppose that $\rho'$ extends to a representation $\overline{\rho}'$ of $\Gamma \times \Gamma \times M$ in $\mathsf{GL}(W)$. Then for every $(g_1,g_2)\in \Gamma \times \Gamma$ and $(h_1,h_2,\ldots, h_{p+2})\in K_0\times K_0 \times K'\times \cdots \times K'$ we argue as in (i) that  $$\overline{\rho}'(g_1,g_2,1,\ldots,1)^{-1}\rho_0'(g_1,g_2,1\ldots, 1)\in \mathsf{GL}(W)$$ centralizes $\rho'(h_1,h_2,\ldots, h_{p+2})$. By irreducibility of $\rho_0'$ on $K_0\times K_0 \times K'\times \cdots \times K'$, there is a group homomorphism $\lambda':\Gamma \times \Gamma \rightarrow \mathsf{GL}_1(\mathbb{C})$ such that $$\overline{\rho}'(g_1,g_2,1,\ldots,1)=\lambda'(g_1,g_2)\rho_0'(g_1,g_2,1\ldots, 1).$$ Since $\overline{\rho}'$ extends $\rho'$ we have $$\lambda'(g_1,g_2)=\pi(g_1,g_2), \ \mbox{ for all } (g_1,g_2)\in L.$$ This is a contradiction, since we showed in the proof of (i) that $\pi$ does not extend to $\Gamma \times \Gamma$.
\end{proof}

\section{Fibre products with exotic finiteness properties and without QI-representations}
\label{sec:Proof-of-Main-results}

In this section we will prove the main results of this paper by combining the results of the previous sections. We start by proving \Cref{thm:Main-fibre-product}. \Cref{thm:Main} will then be a consequence of modifying the construction provided by \Cref{thm:Main-fibre-product} in a way that allows us to vary the finiteness properties and that obstructs the extendability of faithful representations. We will conclude the section with a few remarks, including a second family of examples of groups with exotic finiteness properties, which do not QI-embed in any general linear group (see \Cref{unit-tangent}).

\begin{proof}[{Proof of \Cref{thm:Main-fibre-product}}]
    The existence of lattices $\Gamma<\mathsf{PU}(n,1)$, for every $n>1$, which satisfy the hypotheses follows from \Cref{thm:LIP-existence}. 
    
    We now prove (1). \Cref{thm:coabelian-n-n+1-n+2} implies that $P$ is of type $\mathcal{F}_{2n-1}$. To see that $P$ is not of type $\mathcal{F}_{2n}$, we first observe that $P$ is the kernel of the surjective homomorphism $\Gamma \times \Gamma \to \mathbb{Z}^2$, $(g_1,g_2)\mapsto \phi(g_1)-\phi(g_2)$. Then the assertion follows by applying a Theorem of L\"uck \cite[Theorem 3.3(4)]{Luc-98}, which says that such a kernel is not $\mathcal{F}_{2n}$ if $\Gamma\times \Gamma$ has non-trivial $\ell^2$-Betti number in degree $2n$, where the latter follows from the K\"unneth formula for $\ell^2$-Betti numbers \cite[Theorem 6.54(5)]{Luc-02} and the fact that cocompact lattices in $\mathsf{PU}(n,1)$ have non-trivial $n$-th $\ell^2$-Betti number \cite{Bor-85}.

    Finally, Assertion (2) is a direct consequence of \Cref{thm:Main-fibre-product-2}, since cocompact lattices in $\mathsf{PU}(n,1)$ are hyperbolic. This completes the proof.
\end{proof}

To deduce \Cref{thm:Main} from \Cref{thm:Main-fibre-product} we need the following elementary observation.

\begin{observation}\label{finite-index-qie} Let $P$ be a finitely generated group which does not admit a linear representation, over a local field $k$, which is a quasi-isometric embedding. Then for every finite index subgroup $P'<P$, every linear representation of $P'$ over a local field $k$ is not a quasi-isometric embedding.
\end{observation} 

\begin{proof}
Suppose that the claim is false and that there is a representation $\rho:P'\rightarrow \mathsf{GL}_{d}(k)$, which is a quasi-isometric embedding. By passing to a deeper finite-index subgroup of $P'$, we may assume that $P'$ is normal in $P$. Since $\rho$ is a quasi-isometric embedding, for every $g\in P$, the representation $\rho_g:P'\rightarrow \mathsf{GL}_d(k)$, $h\mapsto\rho_g(h)=\rho(ghg^{-1})$, is also a quasi-isometric embedding. There is an induced representation $\rho':P\rightarrow \mathsf{GL}_{dq}(k)$, $q:=[P:P']$ (see \cite[Section 3.3]{Fulton-Harris}), with the property that its restriction to $P'$ is a direct sum of representations of the form $\{\rho_{g_i}\}_{i=1}^{q}$, for some $g_1,\ldots, g_{q}\in P$. In particular, $\rho'$ restricted to $P'$ (and hence $\rho'$ itself) is a quasi-isometric embedding, which is a contradiction. The observation follows.
\end{proof}

Recall that for $n\in \mathbb{N}$, $SB_n$ denotes the  $n$-th Stallings--Bieri group which is the kernel of a factorwise surjective homomorphism $F_2 \times \cdots \times F_2 \to \mathbb{Z}$ from a direct product of $n$ $2$-generated free groups onto $\mathbb{Z}$. We can now deduce \Cref{thm:Main} from \Cref{thm:Main-fibre-product}.

\begin{proof}[Proof of Theorem \ref{thm:Main}] Let $n>1$, let $\Gamma<\mathsf{PU}(n,1)$ be a torsion-free residually $p$-finite lattice, and $\phi: \Gamma \to \mathbb{Z}^2$ a surjective homomorphism, satisfying the hypotheses of \Cref{thm:Main-fibre-product}. Let us consider the representation $\rho_n:\mathsf{PU}(n,1)\rightarrow \mathsf{GL}_{(n+1)^2}(\mathbb{C})$, $\rho_n([g])=g\otimes \overline{g}, [g]\in \mathsf{PU}(n,1)$, which maps a proximal matrix $[g]\in \mathsf{PU}(n,1)$ to a proximal matrix $\rho_n(g)\in \mathsf{GL}_{(n+1)^2}(\mathbb{C})$. Note that for $[g]\in \mathsf{PU}(n,1)$ we have  $||\mu(\rho_n([g]))||\geq \log\sigma_1(g\otimes \overline{g})= 2\log \sigma_1(g)$. Since $\Gamma$ is quasi-isometrically embedded into $\mathsf{PU}(n,1)$, there are $C,c>0$ such that $\log \sigma_1(\gamma)\geq \frac{1}{2}||\mu(\gamma)||\geq c|\gamma|_{\Gamma}-C$, for every $\gamma\in \Gamma$ and some word metric $|.|_{\Gamma}$ on $\Gamma$; here we use that $\gamma$ is a hyperbolic element of $\mathsf{PU}(n,1)$ and thus $\sigma_1(\gamma)=1/\sigma_{n+1}(\gamma)$ and $\sigma_i(\gamma)=1$ for $i\neq 1,n+1$. Thus the restriction of $\rho_n$ to $\Gamma$ defines a quasi-isometric embedding of $\Gamma$ into $\mathsf{GL}_{(n+1)^2}(\mathbb{C})$, and hence $\Gamma \times \Gamma$ admits a faithful representation into $\mathsf{GL}_{2(n+1)^2}(\mathbb{C})$ which is a quasi-isometric embedding.

The restriction $\tau_n:\mathsf{PU}(n,1)\rightarrow \mathsf{GL}(V)$ of $\rho_n$ on the linear span $V\subset \mathbb{C}^{n+1}\otimes \mathbb{C}^{n+1}$ of the proximal limit set of $\rho_n(\mathsf{PU}(n,1))$ is an irreducible proximal representation restricted to any lattice of $\mathsf{PU}(n,1)$. Since $\Gamma<\mathsf{PU}(n,1)$ is torsion-free and cocompact, all its non-trivial elements are hyperbolic and thus proximal. Thus, we can apply Lemma \ref{non-extension}(i) to the strongly irreducible group $\tau_n(\Gamma)<\mathsf{GL}_d(V)$ and its infinite coabelian subgroup $\tau_n(\textup{ker}(\phi))$. This implies that there is a finite-index subgroup $L<P$ of the fiber product $P=\big\{(g_1,g_2)\in \Gamma\times \Gamma: \phi(g_1)=\phi(g_2)\big\}$ and a faithful irreducible representation $\rho:L\rightarrow \mathsf{GL}_{r}(\mathbb{C})$, $r=\textup{dim}(V)^2$, which does not extend to a representation of $\Gamma\times \Gamma$ into $\mathsf{GL}_r(\mathbb{C})$. Since $L$ has finite index in $P$, we deduce from Theorem \ref{thm:Main-fibre-product} (1) that $L$ is of type $\mathcal{F}_{2n-1}$ and not of type $\mathcal{F}_{2n}$. In addition, by Observation \ref{finite-index-qie}, $L$ fails to admit a representation over a local field which is a quasi-isometric embedding. Thus, if $m=2n$ is even, the ${\rm CAT}(0)$ group $G=\Gamma \times \Gamma$ and its subgroup $L<G$ satisfy the conclusion of the theorem and we are done.

If $m\geq 3$ is odd, then we choose $n\in \mathbb{N}$ with $m=2n-1$, $$G=\Gamma\times \Gamma\times \underbrace{F_2\times \cdots \times F_2}_{m-\textup{times}}$$ and we claim that $L'=L\times SB_{m}<G$ satisfies the conclusion. Since $SB_{m}$ is of type $\mathcal{F}_{m-1}$ and not $\mathcal{F}_{m}$ and $2n-1>m-1$, by Proposition \ref{product-typeF_m}, $L'$ is of type $\mathcal{F}_{m-1}$ and not $\mathcal{F}_{m}$, proving (1). In addition, $L\times SB_m$ does not admit a representation over a local field which is a quasi-isometric embedding, since the quasi-isometrically embedded subgroup $L$ does not, proving (3). Since there is a faithful representation $F_2 \xhookrightarrow{} \mathsf{SL}_2(\mathbb{C})$ which is a quasi-isometric embedding all of whose elements are proximal and whose image is strongly irreducible (e.g. take $F_2$ embedded as a quasi-convex free subgroup of a uniform lattice in $\mathsf{SL}_2(\mathbb{R})$) we conclude that (2) holds.

Now note that $SB_{m}$ contains the product $\gamma_2(F_2)\times \cdots \times \gamma_2(F_2)$. So we can apply Lemma \ref{non-extension} (ii) to $\Delta'\cong F_2$, yielding a faithful irreducible representation $\rho:L\times SB_{m}\rightarrow \mathsf{GL}_{\ell}(\mathbb{C})$, $\ell=2^{m}d^2$, which does not extend to a representation of $G$ into $\mathsf{GL}_{\ell}(\mathbb{C})$. This proves (4), thus completing the proof.\end{proof}

We end this work with a few remarks on further examples and generalizations of our results. We start by recording a second family of examples of groups with arbitrary exotic finiteness properties that do not admit quasi-isometrically embedded representations into any general linear group over a local field. We emphasize that in contrast to our examples, they are not subgroups of ${\rm CAT}(0)$ groups. We thank Sami Douba for pointing this family out to us.

\begin{rmk} \label{unit-tangent} \normalfont{The symmetric space (resp. Bruhat--Tits building) $(X,\mathsf{d})$ associated with $\mathsf{GL}_d(k)$ is ${\rm CAT}(0)$. Thus, a source of examples of groups that fail to admit  quasi-isometric embeddings into any general linear group over a local field are groups $G$ which do not admit quasi-isometric embeddings by representations into the isometry group ${\rm Isom} (X)$ of any complete ${\rm CAT}(0)$ space $X$. Generalising the case of general linear groups over local fields, by the latter we mean a representation $\rho : G \to {\rm Isom}(X)$ such that for every $x\in X$ the orbit map $G\to X$, $g\mapsto \rho(g)\cdot x$ is a quasi-isometric embedding. Note that such groups can clearly not be subgroups of ${\rm CAT}(0)$ groups. A family of such groups is given by groups $G$ which contain an undistorted infinite cyclic subgroup $\langle t\rangle< G$ such that the translation length of $t$ is zero for any action of $G$ on a ${\rm CAT}(0)$ space. To see this, observe that if $|\cdot|_{G}$ is a word metric on $G$, then by undistortedness $\lim_n \frac{1}{n}|t^n|_G>0$. If $\rho: G \rightarrow {\rm Isom(X)}$ were a quasi-isometrically embedded representation, then, by definition, for any base point $x_0\in X$ the orbit map $G \to X$, $x_0 \mapsto g\cdot x_0$ is a quasi-isometric embedding. However, in this case there would be a constant $c>0$ such that the translation length satisfies the inequality
\[
    0= \inf_{x\in X} d(tx,x)\geq \lim_{n \rightarrow \infty} \frac{1}{n}d(t^nx_0,x_0) \geq c \lim_{n\rightarrow \infty} \frac{1}{n}|t^n|_G>0,
\]
where the first inequality does not depend on the choice of $x_0$ (see \cite[II.6.6]{BriHae-99} and also \cite[Lemma 2.1.1]{Douba-thesis}). This is a contradiction.
A concrete example of a group with this property is given by the fundamental group $\pi_1(T^1\Sigma)$ of the unit tangent bundle of a closed hyperbolic surface $\Sigma$, where $t$ is the central element corresponding to the Seifert fibre. The subgroup generated by $t$ is undistorted by \cite[Lemma 3.6.4]{BdHV}, while the fact that the translation length of $t\in \pi_1(T^1\Sigma)$ for any action of $\pi_1(T^1\Sigma)$ on a ${\rm CAT}(0)$ space is zero follows from the proof of \cite[Theorem 2.6]{Bridson-10} (see also \cite[Proposition 4.0.13]{Douba-thesis}). Moreover, it is well-known that $\pi_1(T^1\Sigma)$ is linear and of type $\mathcal{F}$. In particular, one can argue as in the proof of \Cref{thm:Main-fibre-product} that for every $m\in \mathbb{N}$ the direct product $\pi_1(T^1\Sigma)\times SB_m$ with the $m$-th Stallings--Bieri group is a group of type $\mathcal{F}_{m-1}$ and not $\mathcal{F}_m$ that does not quasi-isometrically embed in in any general linear group over a local field.
}
\end{rmk}

By construction, the examples in \Cref{thm:Main} and \Cref{unit-tangent} contain free abelian subgroups of rank at least $2$. For the examples in \Cref{thm:Main-fibre-product} the rank of any such subgroup is at most $2$, while for all other examples it is proportional to the $m$ such that the group is $\mathcal{F}_{m-1}$ and not $\mathcal{F}_m$. This raises the question if there are also examples for arbitrary $m\in \mathbb{N}$ so that the maximal rank of an abelian subgroup is two or even one. The following construction provides such an example for $m=1$.

\begin{rmk}{\normalfont There exist finitely generated subgroups of $\textup{CAT}(0)$ linear hyperbolic groups all of whose linear representations over any local field fail to be quasi-isometric embeddings. Indeed, fix $Q$ a finitely presented group without proper finite index subgroups and $H_2(Q,\mathbb{Z})=0$ (e.g. take $H$ to be the Higman group). Rips construction \cite{Rip-82} applied for $Q$, exhibits a short exact sequence $1\rightarrow N \rightarrow \Gamma \rightarrow Q\rightarrow 1$, where $\Gamma$ is a $C'(\frac{1}{6})$ small cancellation group (hence linear and cubulated by \cite{Ago-13,Wise-2003}) and $N\unlhd \Gamma$ is finitely generated and not finitely presented, thus not quasi-isometrically embedded. By \cite[Proposition 3.3]{BauRei-15}, the pair $(\Gamma,N)$ is a Grothendieck pair (in particular the  inclusion $N\xhookrightarrow{} \Gamma$ induces an isomorphism between profinite completions). By \cite[(1.2) and (1.3)]{Gro-70} (see also \cite[Theorem 4.2]{BasLub-00}), for any local field, any linear representation $\rho:N \rightarrow \mathsf{GL}_d(k)$ extends to a representation $\overline{\rho}:\Gamma \rightarrow \mathsf{GL}_d(k)$. In particular, if $|\cdot|_{\Gamma}$ denotes a word metric on $\Gamma$, there exist $C>1$ such that $||\mu(\rho(g))||=||\mu(\overline{\rho}(g))||\leq C|g|_{\Gamma}+C$ for every $g\in N$. This shows, that since $N$ is distorted in $\Gamma$, $\rho$ fails to be a quasi-isometric embedding.}\end{rmk}

Finally, we would like to point out that there was no loss of generality in restricting to general linear groups, instead of semi-simple Lie groups over local fields.

\begin{rmk} \normalfont{If $\mathsf{G}$ is an arbitrary semi-simple Lie group over a local field $k$, then there is an embedding $\mathsf{G}\hookrightarrow \mathsf{GL}_d(k)$ such that the Cartan projection on $\mathsf{G}$ is the restriction of the Cartan projection on $\mathsf{GL}_d(k)$. In particular, every quasi-isometrically embedded representation of a discrete group $L$ to $\mathsf{G}$ induces a quasi-isometrically embedded representation of $L$ to $\mathsf{GL}_d(k)$. Thus, \Cref{thm:Main}(3) remains true if we replace $\mathsf{GL}_d(k)$ by $\mathsf{G}$.}\end{rmk}

\bibliography{References}
\bibliographystyle{alpha}
\end{document}